\documentclass[reqno,11pt]{amsart}
\usepackage{amsmath,amssymb,amsfonts,amsbsy,setspace, amsthm,anyfontsize}
\usepackage[utf8]{inputenc}

\makeatletter

\everymath{\displaystyle} %pour tout afficher en displaystyle !

 \theoremstyle{plain}
 \newtheorem{thm}{Theorem}[section]
 \newtheorem{lem}[thm]{Lemma}
 \newtheorem{cor}[thm]{Corollary}
 \numberwithin{equation}{section} %% Comment out for sequentially-numbered
 \numberwithin{figure}{section} %% Comment out for sequentially-numbered
 \newtheorem{prop}[thm]{Proposition}
 \theoremstyle{remark}
 \newtheorem{rmk}[thm]{Remark}
 \newtheorem{exa}[thm]{Example}
 \newtheorem*{acknowledgement*}{Acknowledgement}
 \theoremstyle{definition}
 \newtheorem{defn}[thm]{Definition}

\def\eps{\varepsilon}
\newcommand{\R}{\mathbb R}
\newcommand{\Rd}{\mathbb R^{d}}
\newcommand{\C}{\mathcal{C}}
\newcommand{\N}{\mathbb N}

\newcommand{\bV}{\overline V}
\newcommand{\bO}{\overline \Omega}
\newcommand{\pV}{\partial V}
\newcommand{\pO}{\partial \Omega}
\newcommand{\dlambda}{d_{\lambda}}
\newcommand{\Llambda}{L_{\lambda}}
\newcommand{\Smu}{S_{\mu}}
\newcommand{\dlambdaV}{d_{\lambda}^{V}}
\newcommand{\Winf} {W^{1,\infty}}
\newcommand{\A}{\mathcal{A}}

\def\esssup#1{\underset{#1}{\rm ess.sup}\,}
\newcommand{\Leb}{\mathcal{L}}

\newcommand{\path}{\mathsf{path}}

\usepackage{color}
\usepackage{hyperref}
\usepackage{enumerate}
\usepackage{eucal}



\begin{document}
	
\thispagestyle{plain}
\begin{center}
	\Large
	\textbf{A property of Absolute Minimizers in $L^\infty$ Calculus of Variations and of solutions of the Aronsson-Euler equation}
	
	\vspace{0.9cm}
	\large

	\textbf{Camilla Brizzi and Luigi De Pascale}
	
	\vspace{1.5cm}

	\textbf{Abstract} \\
\begin{flushleft}
	We discover a new minimality property of the absolute minimizers of supremal functionals, whose variational problems are also known as $L^\infty$ variational problems. In particular for every minimizer $v$ of the quasi-convex functional \[\esssup{\Omega}H(x,Dv(x))\] we consider the set \[\mathcal{A}(v)=\{x: H (x,Dv(x))=\esssup{\Omega}H(x,Dv(x))\}, \] suitably defined. If $u$ is an absolute minimizer we give a structure result for $\mathcal{A}(u)$ and we show that  then $\mathcal{A}(u)\subset\mathcal{A}(v)$ for every minimizer $v$.
\end{flushleft}
\small
\vspace{1.5cm}
\begin{flushleft}
\vspace{0.2cm}
\noindent{\bf Subjclass:}[2020]{49K30, 49K20, 49N60}
\end{flushleft}

\end{center}

%%%%%%%%%%%%%%%%%%%%%%%%%%%%

\newpage
\title[The attainment set]{A property of Absolute Minimizers in $L^\infty$ Calculus of Variations and of solutions of the Aronsson-Euler equation}

\author{Camilla Brizzi}
\author{Luigi De Pascale}\address{{\bf C.B. \& L.D.P.} Dipartimento di Matematica ed Informatica,
Universit\'a di Firenze, Viale Morgagni, 67/a - 50134 Firenze, ITALY}

\begin{abstract} 
We discover a new minimality property of the absolute minimizers of supremal functionals (also known as $L^\infty$ Calculus of Variations problems). 
\end{abstract}

\date{\today}
\maketitle

%%%%%%%%%%%%%%%%%%%%%%%%%%%%%%%%%%%%%%%%%%%%%%%%%%%%%%%%%%%%%%%%%%%%%%%%%%%%%%%%%%%%%%%%%%%
\section{Introduction}
%%%%%%%%%%%%%%%%%%%%%%%%%%%%%%%%%%%%%%%%%%%%%%%%%%%%%%%%%%%%%%%%%%%%%%%%%%%%%%%%%%%%%%%%%%%

We consider a connected and bounded open set $\Omega$ in $\mathbb{R}^{d}$ and a functional $H:\Omega\times\R^{d}\to \overline{\R}$. The following is an $L^\infty$ problem in Calculus of Variations
\begin{equation}
\label{prob}
\min\left\{F(v,\Omega):=\esssup{x\in\Omega}H(x,Dv(x)):v\in g+W^{1,\infty}(\Omega)\cap \C_{0}(\Omega)\right\},
\end{equation}
where $F$ is called \textit{supremal functional}.\\
These problems have been around since the late '60s, are interesting for several applications and may be more general 
of the one above (for example, $H$ could depend also on $v$). 
\begin{defn}
An absolute minimizer for \eqref{prob} is a function $u\in W^{1,\infty}(\Omega)\cap\C(\bO)$ such that $u=g$ on $\pO$ and for all open subset $V\subset\subset \Omega$ one has 
\begin{equation*}
\esssup{x\in V} H(x,Du(x))\le\esssup{x\in V}H(x,Dv(x))_{}
\end{equation*}
for all $v$ in $W^{1,\infty}(V)\cap\C(\bV)$ such that $u=v$ on $\pV$.
\end{defn}
Absolute minimizers are also well known because they satisfy (in the sense of viscosity) the natural necessary condition for minimisation which is, in this case, 
known as Aronsson-Euler equation and reads as follows
$$ - \frac{d\ }{dx} H(x, Du(x))\cdot D_p H(x, Du(x)) =0.$$
The most well-known among the Aronsson-Euler equations is certainly the $\infty$-Laplacian equation 
$$-\Delta_\infty u= -\langle D^2 u Du, Du \rangle=0 $$
which is associated to the Hamiltonian $\frac{|Du|^2}{2}$.
 
The Aronsson-Euler equation was first discovered by Aronsson in \cite{Aro1965,Aro1966,Aro1967,Aro1968} as necessary condition for absolute minimality. Derivation and sufficiency are discussed in several papers, see, for example, 
\cite{AroCraJuu2004,crandall2009derivation,Cra2003,CraEvaGar2001,yu2006variational}. Existence, uniqueness and regularity of solutions had in the last two decades a tumultuous development. 
A partial view  is given by \cite{ CraEvaGar2001}. 
There is also a game theoretic, probabilistic, averaging point of view which has, more recently, gained popularity. We will not consider that approach here although it deserves attention for many aspects. The vectorial and higher order case are more difficult and for a starting point we may refer to \cite{BarJenWan2001, kat2020, kat2018}.

Throughout this work, we assume the following:
\begin{enumerate}[(A)]
	\item \label{propA} $H\ge0$, $H(\cdot,0)=0$ and $H(x,\cdot)$ is quasi-convex, i.e. any sublevel $\{H(x,\cdot)\le\lambda\}$ is convex;
	\item \label{propB}the map $(x,p)\mapsto H(x,p)$ is uniformly (with respect to $x$) coercive in $p$, which means 
	\begin{equation*}
	\forall \lambda \ \ \exists M\ge 0 \ \mbox{s.t.} \ \ \ H(x,p)\le\lambda \implies |p|\le M;
	\end{equation*}
	\item \label{propC} the map $(x,p)\mapsto H(x,p)$ is continuous in $\Omega\times\mathbb{R}^{d}$; 
	\item \label{propD}For all $\lambda>\mu\ge 0$ 
	%and $V\subset\subset\Omega$ (ho tolto la dipendenza da V e fatto per tutto \Omega)
	there exists $\alpha>0$ such that 
	\begin{equation*}
	\forall x \in \Omega \ \ \left\{H(x,\cdot)\le\mu\right\}+B(0,\alpha)\subset\{H(x,\cdot)\le\lambda\};
	\end{equation*}
	\item \label{propE} For all $\beta>0$, $\lambda>0$ and $V\subset\subset\Omega$, there exists $\delta$ such that, 
\begin{equation*}
\left|\lambda-\overline{\lambda}\right|<\delta\implies\forall x\in V, \ \ \left\{H(x,\cdot)<\lambda\right\}\subset\left\{H(x,\cdot)<\overline{\lambda}\right\}+B(0,\beta).
\end{equation*}
\end{enumerate}
\begin{rmk}\label{remarkpropr}
\begin{enumerate}[(i)]
We observe that: 
\item the continuity of $H$, given by the property \eqref{propC}, allows for the local uniform continuity. This means that for every $x_{0}\in\Omega$ and for every $K,r>0$, there exists a non decreasing function $\omega:[0,\ +\infty)\to[0,\ +\infty) $, such that $ \lim_{t\to 0}\omega (t)=0$
\begin{equation*}
  |H(x,p)-H(y,p')|  \leq \omega(|x-y|+|p-p'|),  
\end{equation*}
for all $(x,p),(y,p')\in B(x_0,r) \times B(0, K)$;
\item by property \eqref{propD},  since $0 \in \{H(x, \cdot) \leq \lambda \}$, for any $\lambda>0$ there exists $\alpha>0$ such that 
\[ B(0,\alpha) \subset \{H(x,\cdot)\le \lambda \}, \ \text{for any }x\in \Omega;   \]
\item property \eqref{propE} implies that the interior part of the level set $\{H(x,p)=\lambda\}$ is empty for every $\lambda\ge 0$;
\end{enumerate}
\end{rmk}

\begin{exa}\label{localD}
The functional $H:\Omega\times\R\to\R$, defined by
\begin{equation*}
H(x,p):=\frac{|p|}{dist(x,\partial \Omega)},
\end{equation*}
satisfies all the properties \eqref{propA}-\eqref{propE}, but the property \eqref{propD}. However it satisfies a local version of such property, that is 
\begin{equation*}
    \forall \, V\subset\subset \Omega,	\forall x \in V \ \ \left\{H(x,\cdot)\le\mu\right\}+B(0,\alpha)\subset\{H(x,\cdot)\le\lambda\}.
\end{equation*}
\end{exa}

We will show that for any minimizer $v$ of problem \eqref{prob} above it is possible to give a point-wise definition of 
\[H(x,Du(x))\]
which will be denoted by $H(x,Dv)(x)$ to distinguish it from the classical value which is only almost everywhere defined.
For any minimizer we introduce the attainment set 
\begin{equation}\label{attain}
\A(v):= \{x \in \Omega \ | \ H(x, Dv)(x)= \esssup{\Omega} H(x,Dv(x))\}.
\end{equation}
and we will prove the following minimality property

\hfill

{\bf Theorem} Let $u \in \Winf(\Omega)\cap \C(\overline \Omega)$
be an absolute minimizer for problem \eqref{prob} then 
\begin{equation}
\A(u)\subset \A(v)
\end{equation}	
for all minimizers $v$ of \eqref{prob}

\hfill

We will also give a qualitative description of $\A(u)$.

%%%%%%%%%%%%%%%%%%%%%%%%%%%%%%%%%%%%%%%%%%%%%%%%%%%%%%%%%%%%%%%%%%%%%%%%%%%%%%%%%%%%%%%%%%%
\section{Level-convex duality and pseudo distances}
%%%%%%%%%%%%%%%%%%%%%%%%%%%%%%%%%%%%%%%%%%%%%%%%%%%%%%%%%%%%%%%%%%%%%%%%%%%%%%%%%%%%%%%%%%%
Absolute minimizers are well characterized in terms of a family of pseudo-distances associated to the  the quasi-convex conjugate of $H$.
 
\subsection{The quasi-convex conjugate of $H$} 
\begin{defn}
For any $x\in\Omega$ and $\lambda\ge 0$, we define $L_\lambda (x,\cdot):\mathbb{R}^{d}\to\mathbb{R}$ by
\begin{equation*}
L_{\lambda}(x,q):=\sup\{p\cdot q:\ p\in\mathbb{R}^d,\  H(x,p)\le \lambda \}.
\end{equation*}
\end{defn}
\begin{rmk}\label{propofL} \hfill 
\begin{enumerate}[(i)]
\item For any $\lambda\ge 0$, $L_{\lambda}$ is a Finsler metric (we refer to \cite{BaoCheShe2000} for more details about Finsler metrics). Indeed, $L_{\lambda}:\Omega\times\R ^d \to\R_{\ge 0}$ is a Borel-measurable function such that $L_{\lambda}(x,\cdot)$ is positively 1-homogeneous and convex for all $x\in\Omega$. 
The measurability of $L_{\lambda}$ is due to the upper semicontinuity w.r.t. $x$ (that follows from the continuity of $H$) and the convexity w.r.t $q$;
\item for any positive $\lambda$, the unitary ball of $L_\lambda (x, \cdot)$ is the polar set of the convex sublevel set $H(x,\cdot)\leq \lambda$ of $H$. 
Indeed by the assumptions \eqref{propA} and \eqref{propC} the set $\{H(x,\cdot)\le\lambda\}$ is convex and closed and then, thanks to Hahn-Banach Separation Theorem, we have the following characterization:
\begin{equation}
\label{charsublevels}
p \ \mbox{is such that} \ H(x,p)\le\lambda \iff \sup\{p\cdot q: L_{\lambda}(x,q)\le 1\}\le 1.
\end{equation}

\item from properties \eqref{propB} and \eqref{propD} we infer that for every $\lambda\ge 0$ there exists $0<\alpha< M$, such that 
\begin{equation}
\label{boundsLlambda}
\alpha|q|\le L_{\lambda}(x,q)\le M|q|, \quad \mbox{for every }x\in\Omega, \ \mbox{and }q\in\Rd.
\end{equation}
\end{enumerate}
\end{rmk}

\begin{lem}\label{lsc} The map $(x,q, \lambda) \mapsto L_\lambda (x,q)$ is lower semicontinuos. 
\end{lem}
\begin{proof}
The lower semicontinuity of $\lambda\mapsto L_\lambda(x,q)$ is due to the property \eqref{propE}.
\end{proof}
\begin{rmk}
Lemma \ref{lsc} together with the Remark \ref{propofL} implies continuity of $L_\lambda(\cdot,q)$. Notice that the continuity of $L_\lambda(x,\cdot)$ is given by convexity.
\end{rmk}
We have also the following result of lower semicontinuity in $x$ of $L_\lambda$, uniform with respect to $q$.
\begin{lem}\label{uniflsc} Let $\lambda >0$, $x_0 \in \bO$, $K>0$. Then for all $ \eta >0$ there exists $ \delta >0$ such that 
$$ |q| \leq K, \ |x-x_0| \leq \delta  \ \ \ \Rightarrow \ \ \ L_\lambda (x,q) \geq L_\lambda (x_0,q) -\eta.$$
\end{lem}
\begin{proof} 
Let us assume by contradiction that there exist $\bar{\eta}>0$ such that for all $n\ge0$, there exists $q_{n}$, with $q_n\le K$, and $x_{n}\in\bO$, with $|x_{n}-x|<2^{-n}$, such that 
\begin{equation*}
L_\lambda(x_{n}-q_{n})-L_\lambda(x_{0},q_{n})<-\bar{\eta}.    
\end{equation*}
Since $K$ is compact, up to the choice of a subsequence, we have that $q_{n}$ converges to a certain point $q\in K$. Then we get the following contradiction:
\begin{equation*}
 L_\lambda(x_{0},q)\le \liminf_{n\to\infty}L_\lambda(x_{n},q_{n})<\liminf _{n\to\infty}
 L_\lambda(x_{0},q_{n})-\bar{\eta}=L_\lambda(x_{0},q)-\bar{\eta},   
\end{equation*}
where the first inequality is due to Lemma \ref{lsc} and the equality follows from the convexity, and then the continuity, of $L_\lambda(x,\cdot)$.

%%%%%%%%%%%%%%%%%%Dimostrazione alternativa meno semplice%%%%%%%%%%%%%%%%%%%%%%
%To prove that 
%$$ L_\lambda (x_0,q) \leq L_\lambda (x,q) +\eta,$$
%it is enough to prove that 
%$$ \{H(x_0,p) \leq \lambda\}\subset \{ H(x,p) \leq \lambda\} + B(0, \frac{\eta}{K}) %.$$
%Indeed if we write $p \in \{H(x_0,p) \leq \lambda\} $ as $ p= \tilde p +v$ for $ %\tilde p \in   \{ H(x,p) \leq \lambda\} $ and $v \in  B(0, \frac{\eta}{K})$ 
%we have, for all $q$ with $|q| \leq K$, 
%$$ q \cdot p  = q \cdot (\tilde p +v) \leq  q \cdot \tilde p + |q||v| \leq   q \cdot %\tilde p + \eta;$$
%taking the $\sup$ on the left-hand side and on the right hand side gives the desired %inequality.
%To prove the inclusion we observe that, by (i) of Remark \ref{remarkpropr}, if %$|x-x_0|$ is small enough 
%$$H(x,p) \leq H(x_0,p)+ |H(x,p)-H(x_0,p)| \leq H(x_0,p) + \omega (|x-x_0|). $$
% Let us choose in property (E) above $\beta =\frac{\eta}{K}$ and the $\delta$ %relative to $\lambda$ and this $\beta$ and consider 
% $x$ such that $\omega(|x-x_0|) \leq \delta$ and $p$ such that $H(x_0,p) \leq %\lambda$, the inequality above gives us 
% $$H(x,p) \leq \lambda + \delta. $$
%This means that
%$$ \{H(x_0,p) \leq \lambda\}\subset \{ H(x,p) \leq \lambda+ \delta\} \subset \{ %H(x,p) \leq \lambda\}+ B(0,\frac{\eta}{K}), $$
%where the last inclusion follows from property (E).
%%%%%%%%%%%%%%%%%%%%%%%%%%%%%%%%%%%%%%%%%%%%%%%%%%%%%%%%%%
\end{proof}

\subsection{The family of pseudo-distances}
 The fact that $L_{\lambda}$ is a Finsler metric for all $\lambda\ge 0$, allows for the definition of a family of pseudo-distances associated to $H$ on $\Omega$ and on the connected open subsets of $\Omega$.
 This family of pseudo-distances is useful to characterize absolute minimizers as we will see in the following section.
   
  \begin{defn}\label{distance}
	For any $x,y\in \Omega$ and any $\lambda\ge 0$, we set
	\begin{equation*}
	d_{\lambda}(x,y):=\inf\left\{\int_{0}^{1}L_{\lambda}(\xi(t),\dot{\xi}(t))dt: \xi\in \mathsf{path}(x,y) \right\},
	\end{equation*}
	where 
	\begin{equation}\label{path}
\mathsf{path}(x,y):=\{\xi\in \Winf((0,1),\Omega)\cap\C([0,1],\Omega):\xi(0)=x, \ \xi(1)=y   \}.
	\end{equation}
	
	For any $x,y\in\overline{\Omega}$ we set 
	\begin{equation*}
	d_{\lambda}(x,y):=\inf\left\{\liminf_{n\rightarrow+\infty}d_{\lambda}(x_{n},y_{n}) : (x_{n})_{n},(y_{n})_{n}\in \Omega^{\mathbb{N}} \ and \ x_{n}\rightarrow x,y_{n}\rightarrow y  \right\}.
	\end{equation*}
\end{defn}

\begin{rmk}\label{propertiesofdistance}\hfill
\begin{enumerate}[(i)]
\item Since the boundary of $\Omega$ is not necessarily regular, one may have $d_{\lambda}(\tilde{x},y)=+\infty$ for some $\tilde{x}\in\pO$ and $y\in \Omega$: in this case, $d_{\lambda}(\tilde{x},y)=+\infty$ for any $y\in \Omega$ due to the connectedness of $\Omega$; 
\item $d_{\lambda}$ is not a priori symmetric, but it satisfies the triangular inequality
\begin{equation*}
d_{\lambda}(x,y)\le d_{\lambda}(x,z)+d_{\lambda}(z,y),
\end{equation*}
for all $x,y\in\overline{\Omega}$ and for all $z\in \Omega$. The inequality may be false for $z\in\pO$\footnote{Consider, for example, $\Omega=\{(x,y) \in \R^2 \ : \ 1 < x^2+y^2<4\  \mbox{and}\ (x,y)\not \in \{0\}\times(-2,-1) \}$, $z \in \{0\}\times(-2,-1)$ and two points in $V$ very close to $z$ but on opposite sides of the segment $ \{0\}\times(-2,-1)$ }. However it holds for every $x,y,z\in\bO$ if $\partial \Omega$ is lipschitz;
\item the definition of $d_{\lambda}$ does not depend on the choice of the domain of the curves $\gamma\in\path$. Indeed, thanks to the absolute $1$-homogeneity of $L_{\lambda}$ with respect to $q$, if $\gamma:[a,b]\to\Omega$ is a Lipschitz curve and $\tilde{\gamma}:[0,1]\to\Omega$ is such that $\gamma(t)=\tilde{\gamma}(\phi(t))$, where $\phi:[a,b]\to[0,1]$ is an increasing reparameterization, we have 
\begin{equation*}
    \int_{a}^{b}L(\gamma(t),\dot{\gamma(t)})dt=\int_{a}^{b}L(\tilde{\gamma}(\phi(t)),\phi'(t)\dot{\tilde{\gamma}}(\phi(t))dt=\int_{0}^{1}L(\tilde{\gamma}(s),\dot{\tilde{\gamma}}(s))ds.
\end{equation*}
\item the inequalities \eqref{boundsLlambda} implies that $d_\lambda  (x,y)$ is equivalent to the intrinsic distance in $\Omega$, that is: for every $x,y\in \bO$
\begin{equation}
\label{equivalenceintrinsicdistance}
\alpha|x-y|_{\Omega}\le d_{\lambda}^{\Omega}(x,y)\le M|x-y|_{\Omega},
\end{equation}
where 
\begin{equation*}
|x-y|_{\Omega}=\inf\left\{\int_{0}^{1}|\dot{\gamma}|dt: \gamma\in\path(x,y) \right\};
\end{equation*}
\item If $\pO$ is Lipschitz then $d_{\lambda}$ is equivalent to the Euclidean distance.
\item  By definition, the function $\lambda\mapsto d_{\lambda}$ is non-decreasing. In particular the property \eqref{propD} of $H$ yields the strict monotonicity. 
\end{enumerate}
\end{rmk}

\begin{rmk}
Sometimes it could be useful to restrict Definition \ref{distance} to open subsets $V$ of $\Omega$, that are connected and well contained in $\Omega$. For example when property \eqref{propD} holds only locally (see Example \ref{localD}). In this case the distance between two points $x,y\in V$ will be
	\begin{equation*}
	d_{\lambda}^{V}(x,y):=\inf\left\{\int_{0}^{1}L_{\lambda}(\xi(t),\dot{\xi}(t))dt: \xi\in \mathsf{path}_{V}(x,y) \right\},
	\end{equation*}
	where 
	\begin{equation*}
\mathsf{path}_{V}(x,y):=\{\xi\in \Winf((0,1),V)\cap\C([0,1],V):\xi(0)=x, \ \xi(1)=y   \}.
\end{equation*}
We point out that all the properties listed in Remark \ref{propertiesofdistance} and all the results in this paper hold also when we restrict to $V$.
\end{rmk}
	\begin{prop}
	\label{continuitydlambda}
The map $\lambda\mapsto d_{\lambda}(x,y)$ is left continuous on $\mathbb{R}_{\ge0}$ for every $(x,y)\in \overline{\Omega}\times \overline{\Omega}$.
\end{prop}
\begin{proof}
	Let's fix $\lambda\ge0$ and consider a sequence $(\lambda_{n})_{n\in\mathbb{N}}\subset\mathbb{R}_{\ge 0}$ such that $\lambda_{n}\rightarrow\lambda$ from the left. Let's then take $\beta>0$. From the property (E) of $H$ we know that there exists $\delta$ such that 
	\begin{equation*}
	|\lambda-\overline{\lambda}|<\delta \ \mbox{implies}\  \left\{H(x,\cdot)<\lambda\right\}\subset\left\{H(x,\cdot)<\overline{\lambda}\right\}+B(0,\beta).
	\end{equation*}
	So, if we take $\bar{n}$ such that $|\lambda_{\bar{n}}-\lambda|<\delta$, we obtain 
	\begin{align}
	\notag&L_\lambda(x,q):=\sup\{p\cdot q: H(x,p)\le\lambda\}\\ 
	\notag&\le\sup\{(\tilde{p}+w_{\beta})\cdot q: H(x,\tilde{p})\le\lambda_{\bar{n}}, \ w_{\beta}\in B(0,\beta) \}\\ 
	\label{continuityL}&= L_{\lambda_{\bar{n}}}(x,q)+\beta|q|,
	\end{align}
	for every $x\in V$. 
	
	From \eqref{continuityL} and  from the monotonicity of $\lambda\mapsto d_{\lambda}$ we infer that, for any $(x,y)\in \Omega\times \Omega$, 
	\begin{equation}
	\label{continuityd}
	d_{\lambda_{n}}(x,y)<d_{\lambda}(x,y)\le d_{\lambda_{n}}(x,y)+\beta |x-y|_{\Omega}.
	\end{equation}
	The left continuity in $\Omega\times \Omega$ is then proved. The extension of this result to any two points $x,y\in \overline{\Omega}\times\overline{\Omega}$ follows by applying the  \eqref{continuityd} to any two converging sequences $x_{n}\to x$ and $y_{n}\to y$ and then considering the infimum.
\end{proof}
\begin{prop}
	\label{continuitydxy}
	The function $(x,y)\mapsto d_{\lambda}(x,y)$ is continuous in $\Omega\times \Omega$ for every $\lambda\ge 0$. Moreover if $\partial \Omega$ is Lipschitz, $d_{\lambda}(x,y)$ is continuous in  $\overline{\Omega}\times\overline{\Omega}$.
\end{prop}
\begin{proof}
 We first notice that the \eqref{equivalenceintrinsicdistance} implies that  
	\begin{equation*}
		\dlambda(x_{n},x)\to 0 \ \mbox{and } \dlambda(x,x_{n})\to 0 \quad \mbox{for every } x\in\Omega \ \mbox{and } (x_n)\subset\Omega: x_{n}\to x.
	\end{equation*}
	Then the thesis follows directly from the triangular inequality, indeed (forgetting the $\lambda$  for lighter notations) for every $x,y\in\Omega$ and $(x_n),(y_n)\subset\Omega$ such that $x_{n}\to x$ and $y_{n}\to y$, 
\[ d(x,y)-d(x_{n},y_{n})\le d(x,x_{n})+d(x_{n},y_{n})+d(y_{n},y)-d(x_{n},y_{n})=d(x,x_{n})+d(y_{n},y) \]
and
	\[ d(x_{n},y_{n})-d(x,y)\le d(x_{n},x)+d(x,y)+d(y,y_{n})-d(x,y)=d(x_{n},x)+d(y,y_{n}).
	\]
	If $\pO$ is Lipschitz the triangular inequality holds also if $x,y\in\partial \Omega$, so $\dlambda$ is continuous till $\pO$. 
	%show that if $u:\overline{V}\to\R$ satisfies, for all $x,y\in V$:
%	\begin{equation*}
%u(y)-u(x)\le d_{\lambda}(x,y)
	%\end{equation*}
	%for some $\lambda\ge 0$, then $u\in\Winf(V)\cap\C(V)$. This  follows by assumption \eqref{propB}. Indeed, if we take $x,y$ close enough, so that the segment $(x,y)$ is included in $V$, then $d_{\lambda}(x,y)\le M |x-y|$. So $u$ is locally Lipschitz in $V$ and $||Du||_{L^{\infty}(V)}\le M$. Then we notice that, since $d_{\lambda}^{V}$ satisfies the triangular inequality, $d_{\lambda}^{V}(x_{0},\cdot)$ satisfies
	%\[d_{\lambda}^{V}(x_{0},y)-d_{\lambda}^{V}(x_{0},x)\le d_{\lambda}^{V}(x,y).\]
	%Finally if $\partial V$ is Lipschitz regular 
	%it yields that $W^{1,\infty}(V)\subset \C(\overline{V})$. This conclude the proof. 
\end{proof}

\section{Absolute minimizers and pseudo-distances}
The following is a ``lipschitz type" characterization of the absolute minimizers. 

\begin{thm}\label{lambdacontinuity}
	Let $u\in W^{1,\infty}(\Omega)\cap\C(V)$. Then $u$ is such that $H(\cdot, Du(\cdot))\le\lambda$ a.e. in $\Omega$ for some $\lambda\ge0$ if and only if for any $x,y\in \Omega$ one has $u(y)-u(x)\le d_{\lambda}(x,y)$. If $u\in\C(\bO)$ then $u(y)-u(x)\le d_{\lambda}(x,y)$ holds for any $x,y\in\bO$.
\end{thm}
Theorem \ref{lambdacontinuity} may be stated in a slightly more general way \cite{ChaDep2007}  but we will not need it here.

\begin{proof}
Let's first assume that $H(\cdot,Du)\le\lambda$ a.e. in $\Omega$ for some $\lambda\ge 0$. 

Let $N:=\{x\in\Omega: u \ \mbox{is not differentiable at }x \ \mbox{or }H(x,Du(x))>\lambda \}$. It then holds that 
\begin{equation*}
u(y)-u(x)\le\inf\left\{\int_{0}^{1}\Llambda(\gamma(t),\dot{\gamma}(t))dt:\gamma\in\mathsf{path}(x,y) \ \mbox{and} \ \gamma \ \mbox{trasversal to} \ N \right\},
\end{equation*}
where trasversal means that $\mathcal{H}^{1}(\gamma((0,1))\cap N)=0$. Indeed, for any $\gamma$ trasversal to $N$, we have 
\begin{equation*}
u(y)-u(x)=\int_{0}^{1}(u\circ\gamma)'dt=\int_{0}^{1}Du(\gamma(t))\cdot\dot{\gamma}(t)dt\le\int_{0}^{1}\Llambda(\gamma(t),\dot{\gamma}(t))dt,
\end{equation*}
where the inequality follows from the definition of $\Llambda$, since $H(\cdot,Du)\le\lambda$. Let's now take $\gamma\in\mathsf{path}(x,y)$. It is possible to approximate $\gamma$ in  $\Winf((0,1))$ by a sequence $(\gamma_{k})_{k}\subset\mathsf{path}(x,y)$, with $\gamma_{k}$ trasversal to $N$ for any $k\in\mathbb{N}$ (see the Appendix, Proposition \ref{trasversal}). Then it follows that
\begin{equation}
\label{necessary} 
u(y)-u(x)\le\inf\left\{\int_{0}^{1}\Llambda(\gamma(t),\dot{\gamma}(t))dt:\gamma\in\mathsf{path}(x,y) \right\}=d_{\lambda}(x,y).
\end{equation}
If $u\in\mathcal{C}(\bar{\Omega})$, then, by \eqref{necessary} and continuity of $u$, one has
\begin{equation*}
u(y)-u(x)\le\liminf_{n\to+\infty}d_{\lambda}(x_{n},y_{n}),
\end{equation*}
 for any $(x_{n})_{n},(y_{n})_{n}$ such that $x_{n}\to x$ and $y_{n}\to y$.
Then $u(y)-u(x)\le d_{\lambda}(x,y)$ for any $x,y\in\bar{\Omega}$.\\
Suppose now that $u(y)-u(x)\le d_{\lambda}^{V}(x,y)$, for any $x,y\in\Omega$. Since $u\in\Winf(\Omega)\cap\mathcal{C}(\Omega)$ 
%(see the first part of the proof of Proposition \ref{continuitydxy}) 
it is differentiable a.e. in $\Omega$. We will show that $H(x_{0},Du(x_{0}))\le \lambda$ for any $x_{0}$ that is a point of differentiability.
We observe that, by \eqref{charsublevels} it is sufficient to show that $Du(x_{0})\cdot q\le 1$ for every $q$ such that $L_{\lambda}(x,q)\le 1$.
\begin{equation*}
Du(x_{0})\cdot q=\lim\limits_{h\to 0}\frac{u(x_{0}+hq)-u(x_{0})}{h}\le\lim\limits_{h\to 0}\frac{d_{\lambda}(x_{0},x_{0}+hq)}{h}.
\end{equation*}
Moreover, for $h$ small enough, by the $1$-homogeneity of $L_{\lambda}$ we have 
\begin{equation*}
\frac{1}{h}d_{\lambda}(x_{0},x_{0}+hq)\le\frac{1}{h}\int_{0}^{1}\Llambda(x_{0}+thq,hq)dt=\int_{0}^{1}\Llambda(x_{0}+thq,q)dt.
\end{equation*}
Finally, by the upper semicontinuity of $L_{\lambda}$ it follows that  
\begin{align*}
\lim\limits_{h\to 0}\frac{d_{\lambda}(x_{0},x_{0}+hq)}{h}\le\limsup_{h\to 0}\int_{0}^{1}\Llambda(x_{0}+thq,q)dt\le\int_{0}^{1}\Llambda(x_{0},q)dt\le 1.
\end{align*}
By \eqref{charsublevels}, we have $H(x_{0},Du(x_{0}))\le\lambda$.
\end{proof}

We now state some useful results from \cite{ChaDep2007}.

\begin{defn}
	Let $g$ be a function in $W^{1,\infty}(\Omega)\cap \C(\bO)$ and $\lambda\ge 0$. We define the functions $S^{-}_{\lambda}(g)$ and $S^{+}_{\lambda}(g)$ given on $\bO$ by:
	\begin{align*}
	&\forall x\in\overline{\Omega} \ \ S^{-}_{\lambda}(g)(x)=\sup\left\{g(y)-\dlambda(x,y): y\in\partial \Omega\right\},\\
	&\forall x\in\overline{\Omega} \ \ S^{+}_{\lambda}(g)(x)=\inf\left\{g(y)+\dlambda(y,x): y\in\partial \Omega\right\}.
	\end{align*}
When $g$ is clear from the context, we will write $S^\pm_\lambda(x)$ instead of $S^\pm_\lambda(g)(x)$.
\end{defn}
\begin{thm}
	\label{optimalsolutionsprob}
Let $g$ be a function of $W^{1,\infty} (\Omega)\cap\C(\bO)$ and consider the problem 
\begin{equation}
\label{probV}
\min\left\{F(v):=\esssup{x\in \Omega}H(x,Dv(x)):v\in g+W^{1,\infty}(\Omega)\cap\C_{0}(\Omega)\right\}.
\end{equation}
Then the minimal value of this problem is
\begin{equation*}
\mu:=\min \{\lambda: g(y)-g(x)\le\dlambda(x,y) \ \mbox{for any } x,y\in\pO \}.
\end{equation*}
Moreover, the functions $S_{\mu}^{-}(g)$ and $S_{\mu}^{+}(g)$ are optimal solution of \eqref{probV} and for any optimal solution $u$ of \eqref{probV} one has
\begin{equation*}
	S_{\mu}^{-}(x)\le u(x)\le S_{\mu}^{+}(x),
\end{equation*} 
for all $x\in\bO$.
\end{thm}
\begin{proof}

We first notice that the minimum $\mu$ does not need \textit{a priori} to be attained, so we first set:
\begin{equation}\label{defmu}
\mu:=\inf \{\lambda: g(y)-g(x)\le\dlambdaV(x,y) \ \mbox{for any } x,y\in\pV \}
\end{equation}
and for every $x\in\bO$,
	\begin{align*}
&\Smu^{-}(x):=\sup\left\{g(y)-\dlambda(x,y):\lambda>\mu, \ y\in\partial V\right\},\\
&\Smu^{+}(x):=\inf\left\{g(y)+\dlambda(y,x):\lambda>\mu, \ y\in\partial V\right\}.
\end{align*}
We now claim that $\Smu^{-}(x)=g(x)$ for any $x\in\pO$. Indeed, taking $y=x$ in the definition of $\Smu^{-}$ yields $\Smu^{-}(x)\ge g(x)$, while by definition of $\mu$ one has $g(y)-d_{\lambda}(x,y)\le g(x)$, for any $\lambda>\mu$ and $y\in\pO$, so that $\Smu^{-}(x)\le g(x)$, which in turn proves the claim. The same holds for $S^{+}$.\\
The second claim is that, for any $\lambda>\mu$ and for any $x,y\in\bO$, one has:
\begin{equation}
\label{estimateSminus}
\Smu^{-}(y)-\Smu^{-}(x)\le d_{\lambda }(x,y).
\end{equation} 
Indeed, take $\lambda>\mu$, $x\in\bO$ and $y\in\Omega$. We notice that since $\lambda\mapsto d_{\lambda}$ is not decreasing, the supremum in the definition of $\Smu^{-}$ can be taken for $\sigma\in(\mu,\lambda]$, so that 
\begin{align*}
&\Smu^{-}(y)-\Smu^{-}(x)=\sup_{z\in\pO,\mu<\sigma\le\lambda}\inf_{z'\in\pO,\mu<\sigma'\le\lambda}\{g(z)-d_{\sigma}(y,z)-g(z')+d_{\sigma'}(x,z')\}\\
&\le\sup_{z\in\pO,\mu<\sigma\le\lambda}\{g(z)-d_{\mu}(y,z)-g(z)+d_{\sigma}(x,z)\}\le \sup_{\mu<\sigma\le\lambda}d_{\sigma}(x,y)=d_{\lambda}(x,y),
\end{align*}
where the last inequality is due to the triangular inequality that holds since $y\in\Omega$.\\
When $y\in\pO$, we notice that  
\[ \Smu^{-}(y)=g(y) \ \text{and} \ \Smu^{-}(x)\ge g(y)-d_{\lambda}(x,y) \quad \text{for every}\lambda>\mu. \]
In a similar way the estimate \eqref{estimateSminus} can be proved for $\Smu^{+}$. \\
By Theorem \ref{lambdacontinuity} we infer that $F(\Smu^{-}),F(\Smu^{+})\le\lambda$ for every $\lambda>\mu$. It follows that $F(\Smu^{-}),F(\Smu^{+})\le\mu$. In particular, by Theorem \ref{lambdacontinuity}, we have that for every $x,y\in\pO$ 
\begin{equation*}
\Smu^{-}(x)-\Smu^{-}(y)\le d_{\mu}(x,y),
\end{equation*}
and, since $\Smu^{-}=g$ on $\pO$ this implies that the minimum in the definition of $\mu$ \eqref{defmu} is attained.\\
Moreover, we show that the minimal value for the problem \eqref{probV} is equal to $\mu$. By contradiction assume that there exists a function $u\in g+W^{1,\infty}(\Omega)\cap \C_{0}(\Omega)$ such that $F(u)\le\lambda$ for some $\lambda<\mu$. Then, by Theorem \ref{lambdacontinuity} again, we would have $u(y)-u(x)\le\dlambda(x,y)$ for all $x,y\in\bO$. Since $u=g$ on $\pO$, this contradicts the minimality of $\mu$.\\
Finally, let $u$ be an optimal solution of \eqref{probV}, i.e. $H(\cdot,Du(\cdot))\le\mu$ a.e. in $\Omega$. Then, by Theorem \ref{lambdacontinuity}, $u(y)-u(x)\le d_{\mu}(x,y)$ for any $x,y\in\bO$. If $y\in\pO$, this yields $g(y)-d_{\mu}(x,y)\le u(x)$ and then $\Smu^{-}(x)\le u(x)$. With similar arguments can be shown that $u\le\Smu^{+}$ in $\bO$. 
\end{proof}

\begin{rmk}\label{compdist}[Comparison with distance functions] Theorem \ref{optimalsolutionsprob} is the key for the so called comparison with distance functions principle (first defined in \cite{ChaDep2007}) which we present here in the 
simplified version that will be used. 

For a positive $\lambda$,  $\alpha\in \R$ and $x_0 \in \Omega$ we consider the distance functions $x \mapsto d_\lambda (x_0, x) + \alpha$ and $x \mapsto -d_\lambda (x, x_0) + \alpha$. 
These functions have the property that if $U \subset \subset \Omega$ and 
$x_0 \not \in U$ then 
\[ d_\lambda (x_0, x) + \alpha \geq S^+ (d_\lambda (x_0, x) + \alpha, U)\ \mbox{on} \ \overline U;\] 
 \[-d_\lambda (x, x_0) + \alpha \leq S^- (-d_{\lambda} (x, x_0) + \alpha, U)\ \mbox{on} \ \overline U.\]
These inequalities implies that if $u$ is an absolute minimizer, such that 
\[ \esssup{U} H(x, Du(x)) \leq \lambda, \]
and 
\[ u \leq  d_\lambda (x_0, x) + \alpha \ \mbox{on} \ \partial U, \]
then 
\[ u \leq  d_\lambda (x_0, x) + \alpha \ \mbox{on} \ \overline U, \]
and the reverse comparison holds for the other distance function.
\end{rmk}

%%%%%%%%%%%%%%%%%%%%%%%%%%%%%%%%%%%
%%%%%%%%%%%%%%%%%%%%%%%%%%%%%%%%%%%
\section{Point-wise definition of $H(x,Du(x))$}
%%%%%%%%%%%%%%%%%%%%%%%%%%%%%%%%%%%
%%%%%%%%%%%%%%%%%%%%%%%%%%%%%%%%%%%
Since $u \in W^{1,\infty} (\Omega)$ the quantity $H(x,Du(x))$ is, a priori, defined only for a.e. $x\in \Omega$. In this section we show that there exists a natural pointwise 
definition of $H(x,Du(x))$ which will be denoted by $H(x,Du)(x)$.

\begin{defn}
	Let $u\in \Winf(\Omega)\cap \C(\bO)$. For any $x_{0}\in\Omega$ and for any $r>0$ such that $r<dist(x_{0},\partial \Omega)$, we set
	\begin{equation}
	\mu(x_{0},r):=\inf\{\lambda: u(x)-u(x_{0})\le d_{\lambda} (x_0,x) \ \mbox{for any} \  x\in B(x_{0},r)\}.
	\end{equation}
	We observe that $\mu(x_{0},r)$ is not decreasing in $r$. This allows for the following definition:
	\begin{equation*}
	H(x_{0},Du)(x_{0}):=\lim_{r\rightarrow 0}\mu(x_{0},r)= \inf_r \mu (x_0,r).
	\end{equation*}
\end{defn}
\begin{lem}\label{muusc} Let $r>0$ then $x \mapsto \mu (x,r)$ is upper semicontinuous in $\Omega$. 
\end{lem}
\begin{proof} Let $x_n \to x$ and we may assume, without loss of generality that
\[ \mu(x_n, r) \to \nu.\]
We want to prove that $\mu(x,r) \geq \nu$.
Let $\alpha <\nu$ and $\alpha'= \frac{\alpha+\nu}{2}$ so that $\alpha < \alpha' < \nu$. 
For $n$ such that $\mu(x_n,r)>\alpha'$ there exists $y_n \in B(x_n,r) $ such that 
\[u(y_n)-u(x_n) > d_{\alpha'} (x_n,y_n),\]
and we may assume, up to extraction of a subsequence, that $y_n \to y \in B(x,r)$.
Using the continuity of $u$ and the properties (continuity and monotonicity) of the pseudo-distances
\begin{equation*}
u(y)-u(x) = \lim_{n \to +\infty} u(y_n)-u(x_n) \geq \lim_{n \to +\infty} d_{\alpha'} (x_n,y_n) =  d_{\alpha'} (x,y) >d_{\alpha} (x,y),
\end{equation*}
which implies $\mu(x,r) >\alpha$ for every $\alpha <\nu$.
\end{proof}
\begin{cor} The function $x \mapsto H(x,Du)(x)$ is upper semicontinuous in $\Omega$. 
\end{cor}
\begin{proof}
This follows from the Lemma \ref{muusc} above since the $\inf$ of upper semicontinuous functions is upper semicontinuous. 
\end{proof}
\begin{prop}
	\label{punctualdeffirstineq}
Let $u\in\Winf(\Omega)\cap \C(\bO)$ be such that $u$ is differentiable at $x_{0}$. Then 
\begin{equation*}
H(x_{0},Du(x_{0}))\le H(x_{0},Du)(x_{0}).
\end{equation*}
\end{prop}
\begin{proof}
Let $(\eps_{r})_{r>0}$ a sequence of positive real numbers that decreases to zero. We want to show that $H(x_{0},Du(x_{0}))\le \mu(x_{0},r)+\eps_{r}$, for every $r>0$. If we fix $\mu_{\eps_{r}}:=\mu(x_{0},r)+\eps_{r} $ and we consider $q\in\Rd$ such that  $L_{\mu_{\eps_{r}}}(x_{0},q)\le 1$, by the characterization \eqref{charsublevels} it is sufficient to show that
\begin{equation*}
Du(x_{0})\cdot q \le 1.
\end{equation*}
Since $u$ is differentiable at $x_{0}$ we have:
\begin{align*}
&Du(x_{0})\cdot q=\lim_{h\to 0}\frac{u(x_{0}+hq)-u(x_{0})}{h}\le\lim_{h\to 0}\frac{1}{h}d_{\mu_{\eps_{r}}}(x_{0},x_{0}+hq)\le\\
&\le\lim_{h\to 0}\frac{1}{h}\int_{0}^{1}L_{\mu_{\eps_{r}}}(x_{0}+thq,hq)dt=\lim_{h\to 0}\int_{0}^{1}L_{\mu_{\eps_{r}}}(x_{0}+thq,q)dt=\\
& = L_{\mu_{\eps_{r}}}(x_{0},q)\leq 1,
\end{align*}
where the first inequality follows from the definition of $\mu(x_{0},r)$. The proof is then concluded.
\end{proof}
\begin{prop}
	Let $u\in\Winf(\Omega)\cap \C(\bO)$ be differentiable at $x_{0}$. 
Then 
	\begin{equation*}
	H(x_{0},Du(x_{0}))\ge H(x_{0},Du)(x_{0}).
	\end{equation*}
\end{prop}
\begin{proof}
For every $r>0$ and such that $B(x_0,r) \subset \subset \Omega$  denote  by $\mu_r= \mu(x_0,r)$ and by $\mu:= H(x_{0},Du)(x_{0})$. Let $\eps>0$ and $x_r \in B(x_0,r)$ be such that 
$$u(x_r)-u(x_0)\ge d_{\mu_r-\eps} (x_0,x_r).$$
Consider a sequence $\frac{x_{r_n}-x_0}{s_n}$, where $s_{n}:=|x_{r_{n}}-x_{0}|$. By compactness up to the choice of a subsequence we have that
$$\frac{x_{r_n}-x_0}{s_n}\to q, $$ 
for some unitary vector $q$. Let $(\eps_n)$ be sequence of real numbers decreasing to $0$. For every $n\ge 0$, we denote by $\mu_n=\mu_{r_n}$. Let $\gamma_{n}\in\mathsf{path}(x_0,x_{r_n})$ such that (up to an increasing reparametrization) $d_{\mu_{n}-\eps}(x_{0},x_{r_{n}})\ge\int_0^{s_n} L_{\mu_n-\eps} (\gamma_n, \dot \gamma_n ) dt-\frac{\eps_n}{2}$, Then 
\begin{equation*}
\begin{split}
\frac{u(x_{r_n})-u(x_0)}{s_n} &\ge  \frac{d_{\mu_n-\eps}(x_0,x_{r_n})}{s_n}\\ &\ge \frac{1}{s_n}\int_0^{s_n} L_{\mu_n-\eps} (\gamma_n, \dot \gamma_n ) dt-\frac{\eps_n}{2} \\
&\geq  \frac{1}{s_n}\int_0^{s_n} L_{\mu-\eps} (\gamma_n, \dot \gamma_n ) dt -\frac{\eps_n}{2}  \\& \geq  \frac{1}{s_n}\int_0^{s_n} L_{\mu-\eps }(x_0, \dot \gamma_n ) dt - \frac{\eps_n}{2}-\frac{\eps_n}{2} \\
& \geq L_{\mu-\eps} (x_0, \frac{1}{s_n} \int_0^{s_n} \dot \gamma _n dt ) - \eps_n \stackrel{\liminf}{\to} L_{\mu-\eps} (x_0,q),
\end{split}
\end{equation*}
obtaining that $Du(x_{0})\cdot q\ge L_{\mu-\eps}(x_{0},q)$. The inequality in the third line above follows from Lemma \ref{uniflsc} with $\eta=\frac{\eps_n}{2}$ 
%since we have that $|\gamma_n (t)-x_0| \leq C s_n$ and $|\dot \gamma_n| \leq C$, for some constant $C>0$.
 It follows by the definition of $L_{\lambda}$ that $$H(x_{0},Du(x_{0}))\ge \mu -\eps.$$
The thesis follows by the arbitrariness of $\eps$.
\end{proof}
\begin{cor} For all $u\in W^{1,\infty} (\Omega) \cap \C(\overline \Omega)$
	\[
	\esssup{\Omega} H(x,Du(x)) =\esssup{\Omega} H(x,Du)(x).
	\]
\end{cor}

%%%%%%%%%%%%%%%%%%%%%%%%%%%%%%%%%%%%%%%%%%%%%%%%%%%%
%%%%%%%%%%%%%%%%%%%%%%%%%%%%%%%%%%%%%%%%%%%%%%%%%%%%
\section{The attainment set and its minimality}
%%%%%%%%%%%%%%%%%%%%%%%%%%%%%%%%%%%%%%%%%%%%%%%%%%%%
%%%%%%%%%%%%%%%%%%%%%%%%%%%%%%%%%%%%%%%%%%%%%%%%%%%%
\begin{prop}\label{chain}
Let $u$ be an absolute minimizer for the problem \eqref{prob}. Let $x_0$ be such that $H(x_0, Du)(x_0)=\esssup{x\in \Omega} H(x,Du(x)):=\mu$ and let $\alpha$ and $M$ be such that $B(0, \alpha) \subset \{H(x, \cdot) \leq \lambda\} \subset B(0, M)$ for all $\lambda$ in a neighborhood of $\mu$. Then there exists $y$ such that $d_\mu (x_0,y)=\frac{\alpha^2}{2M} d (x_0, \partial \Omega)$ such that 
$u(y)-u(x_0)= d_\mu (x_0,y)$ and $H (y, Du)(y)= \mu$. 
\end{prop}
\begin{proof}  Let $\lambda < \mu$, by definition of $\mu$, for all $r>0$ $\exists y_r \in B(x_0,r)$ such that 
\[ u(y)-u(x_0) >d_ \lambda (x_0,y) 
\]
or, equivalently, for every open set $A \subset \subset \Omega$ $\exists y_A \in A$ such that 
\[u(y_A)- d_\lambda (x_0, y_A) > u(x_0).
\] 
Let $R= \frac{\alpha^2}{2M}d (x_0, \partial \Omega)$. Consider the open set $V_R=\{z \in \Omega \ : \ d_\lambda(x_0,z) <R\}$. By the choice of $R$, $V_R \subset \subset \Omega$\footnote{In fact if $z\in V_R$ then $|x_0-z| \leq \frac{1}{\alpha} d_\lambda (x_0, z) < \frac{1}{\alpha}R = \frac{\alpha}{2M}d (x_0, \partial \Omega)<d (x_0, \partial \Omega)$.}. Let 
\[ a= \max_{\overline V_R} \{u(z)-  d_\lambda (x_0,z)\}.
\]
We claim that 
\[ a= \max_{\partial V_R} \{u(z)-  d_\lambda (x_0,z)\}.
\]
In fact, let 
\[ a_b= \max_{\partial (V_R\setminus \{x_0\})} \{u(z)-  d_\lambda (x_0,z)\},
\]
we have that the distance function
\[ a_b + d_\lambda (x_0, z)\]
satisfies  
\[ a_b + d_\lambda (x_0, z)\geq u(z) \ \ \ \mbox{on} \ \ \  \partial (V_R \setminus \{x_0\}), \]
and then, by the minimality of $u$  (see Remark \ref{compdist})
\[ a_b + d_\lambda (x_0, z)\geq u(z) \ \ \ \mbox{on} \ \ \  \overline V_R . \]
Moreover, by the choice of $x_0$ and the fact that $\lambda<\mu$, the maximum in $a_b$ can not be reached in $x_0$. Thus we have $a=a_b$.
Let $y_\lambda\in \partial V_R$  be such that $a=u(y_\lambda)- d_\lambda (x_0,y)$ let $\eps <<1$ and let 
$y_\eps \in B(y_\lambda, \eps)$ be a point of the $d_\lambda-$geodesic between $x_0$ and $y_\lambda$ (which exists by Proposition \ref{a3} in the appendix). By the definition of $V_R$, $y_\eps \in V_R$ so that 
\begin{equation*}
\begin{array}{ll} 
u(y_\eps) -d_\lambda (x_0,y_\eps) & \leq a \leq u(y_\lambda)-d_\lambda (x_0,y_\lambda) \\
&= u(y_\lambda)-d_\lambda (x_0,y_\eps)-d_\lambda (y_\eps, y_\lambda).
\end{array}
\end{equation*}
We obtain 
\[ u(y_\lambda)-u(y_\eps) \geq d_\lambda (y_\eps, y_\lambda),
\]
and then for small $\delta >0$ 
\[ \mu(y_\eps, \eps) > \lambda -\delta 
\]
and when $\eps < \delta $ 
\[ \mu (y_\eps, \delta) \geq \mu (y_\eps,\eps) > \lambda -\delta. 
\]
By the upper semicontinuity (see Lemma \ref{muusc}) we have for fixed $\delta$ 
\[ \mu (y_\lambda, \delta) \geq \limsup_{\eps \to 0} \mu(y_\eps, \delta) \geq \lambda -\delta, \]
and letting $\delta \to 0$, 
\[H(y_\lambda, Du)(y_\lambda) \geq \lambda.\]
Consider $\lambda_n \nearrow \mu$ and the corresponding $y_{\lambda_n} \in \partial V_{R_n}$ as constructed above. 
Up to subsequences $y_{\lambda_n} \to y$ and, by the upper semicontinuity of $H(x, Du)(x)$ we obtain
\[H(y, Du)(y) \geq \mu.\]
Moreover
\[ u(y)-u(x_0)= \lim_{n \to \infty} u(y_{\lambda_n}) -u(x_0) = \lim_{n \to \infty} d_{\lambda_n} (x_0,y_n) = d_\mu (x_0,y),\]
where the last equality is due to Lemma \ref{continuitydlambda}.
\end{proof}
\begin{rmk}\label{remarksequence}
We observe that from Proposition \ref{chain} one can infer that for every $V\subset\subset\Omega$ that contains $x_{0}$, there exists $y\in\partial V$ such that $H(y,Du)(y)=\mu$.
\end{rmk}
%%%%%%%%%%%%%%%%%%%%%%%%%%%%%%%
\begin{thm}\label{sequence}
Let $u$ be an absolute minimizer for the problem \eqref{prob} and let $x_{0}\in \Omega$ be such that
\begin{equation}\label{valoremax}
H(x_{0},Du)(x_{0})=\esssup{x\in \Omega} H(x,Du(x)):=\mu.
\end{equation}
Then there exist $x_{+\infty}, x_{-\infty}\in\partial \Omega$, such that
\[
u(x_{+\infty})-u(x_{-\infty})=d_{\mu} (x_{-\infty},x_{+\infty}),
\]
and 
\[ d_{\mu} (x_{-\infty},x_{+\infty})=d_{\mu} (x_{-\infty},x_0)+d_{\mu} (x_0,x_{+\infty}).\]
\end{thm}
\begin{proof}
We first claim that there exists $x_{+\infty}\in\partial \Omega$, such that 
\begin{equation}
\label{firstclaim}
u(x_{+\infty})-u(x_{0})\ge d_{\mu}(x_{0},x_{+\infty}).
\end{equation}
Starting from $x_0$ and following Proposition \ref{chain} we build a sequence $(x_n) \subset \Omega$ such that 
for all $n \in \N$ 
\[\frac{\alpha^2}{2M^2}d (x_n, \partial \Omega)= \frac{d_\mu (x_n, x_{n+1})}{M} \leq  |x_{n+1}-x_n|  \leq \frac{d_\mu (x_n, x_{n+1})}{\alpha} = \frac{\alpha}{2M}d (x_n, \partial \Omega) \]
and
\[ u(x_{n+1})-u(x_n)=d_\mu (x_n, x_{n+1}).\]
From this it follows that for any $n\in \N$ 
\[
u(x_n)-u(x_0) =\sum_{i=0}^{n-1} d_\mu (x_i, x_{i+1}) \geq \sum_{i=0}^{n-1} \frac{\alpha^3}{2M^2}d (x_i, \partial \Omega)
\]
Since $u$ is continuous in $\overline \Omega$  both series on the right of the equation above converge and then 
\[ \lim_{n \to \infty} d (x_n, \partial \Omega) =0,
\]
and $(x_n)$ is a Cauchy sequence converging to some point $x_{+\infty} \in \partial \Omega$. 
For this point it holds 
\begin{equation*}
\begin{split}
u(x_{+\infty})-u(x_0)= \lim_{n \to \infty} u(x_n)-u(x_0) = \lim_{n \to +\infty} \sum_{i=0}^{n-1} d_\mu (x_i, x_{i+1}) \\
\geq \lim_{n \to +\infty} d_\mu (x_0, x_n) = d_\mu (x_0, x_{+\infty}).
\end{split}
\end{equation*}
In a similar way we can find $x_{-\infty}\in\partial \Omega$, such that 
\begin{equation*}
u(x_{0})-u(x_{-\infty})\ge d_{\mu}(x_{-\infty},x_{0}),
\end{equation*}
and so we get
\begin{equation*}
u(x_{+\infty})-u(x_{-\infty})\ge d_{\mu}(x_{-\infty},x_{+\infty}).
\end{equation*}
One also deduces that 
\begin{align*}
	&d_\mu(x_{-\infty},x_{0})+d_\mu(x_{0},x_{+\infty})\\&=u(x_{+\infty})-u(x_{0})+u(x_{0})-u(x_{-\infty})\\&=u(x_{+\infty})-u(x_{-\infty})=d_\mu (x_{-\infty},x_{+\infty}),
	\end{align*}
	that is $x_{0}$ belongs to a geodesics for $d_\mu$ connecting $x_{-\infty}$ to $x_{+\infty}$.

\end{proof}
In the following example we show that assumption \eqref{propE} on the functional $H$ and the consequent left continuity of the map $\lambda\mapsto d_{\lambda}$ (Prop. \ref{continuitydlambda}), are essential for the validity of Proposition \ref{chain} and the consequent proof of the Theorem \ref{sequence}.
\begin{exa}
	Let's take $\Omega=B(0,2)$ and
	\begin{equation*}
	H(x,p)=\begin{cases}
	|p| \ \ \ \, \ \ \ \ \mbox{if} \ |p|<\frac{1}{2}\\
	\frac{1}{2} \ \ \ \ \ \ \ \ \ \mbox{if} \ \frac{1}{2}\le|p|\le\frac{3}{4}\\
	|p|-\frac{1}{4} \ \ \mbox{if} \ |p|>\frac{3}{4 }
	\end{cases}
	.
	\end{equation*}
	Let $u: B(0,2)\to\mathbb{R}$ be such that $u(x)=\alpha\cdot x$, where $\frac{1}{2}<|\alpha|<\frac{3}{4}$. Then $u$ is an absolute minimizer for $H$ and $H(x,Du(x))=\esssup{x\in\Omega}H(x, Du(x))=\frac{1}{2}$ for all $x\in\Omega$. \\ Calling $x_{0}=0$ and taking $V=B(0,1)$ we have that, for any $y\in\partial V$, 
	\begin{equation*}
	u(y)-u(x_{0})=\alpha\cdot y\le |\alpha|<\frac{3}{4}.
	\end{equation*}
	Moreover  
	\begin{equation*}
	   L_{\frac{1}{2}}(x,q)=\frac{3}{4}|q| \quad \text{for any} x\in\Omega
	\end{equation*}
	and 
	\begin{equation}
	d_{\frac{1}{2}}^{V}(x_{0},y)=\inf\left\{ \frac{3}{4}\int_{0}^{1}  |\dot{\gamma}(t)|dt    \right\}=\frac{3}{4}\int_{0}^{1} |y|dt=\frac{3}{4}
	\end{equation}
	That is
	\begin{equation*}
	\max_{x\in \partial V}\{u(y)-d_{\frac{1}{2}}^V(x,y)\}<0.
	\end{equation*}
\end{exa}

We finally prove the minimality property announced at the end of the introduction  
	\begin{thm} Let $\Omega$ be a bounded open set of $\mathbb{R}^{d}$ and $u\in W^{1,\infty}(\Omega)\cap \C (\overline{\Omega})$
be an absolute minimizer for \eqref{prob} and let $v$ be any other minimizer for \eqref{prob}, then 
\begin{equation}
\A(u) \subset \A(v).
\end{equation}
Moreover $\A(u)$ is the union of Lipschitz curves of minimal length for $d_\mu$ where $\mu$ is the minimal value in \eqref{prob}. 
\end{thm}
\begin{proof} Let $x_ 0\in \A(u)$ then, by Theorem \ref{sequence}, there exist $x_{+\infty}, x_{-\infty}\in\partial \Omega$, such that
	\begin{equation*}
	u(x_{+\infty})-u(x_{-\infty})=d_{\mu}(x_{-\infty},x_{+\infty}).
	\end{equation*}
\[ u(x_{+\infty})-u(x_0)=d_{\mu}^{\Omega}(x_0,x_{+\infty}), \ \mbox{and}\ u(x_0)-u(x_{-\infty})=d_{\mu}^{\Omega}(x_{-\infty},x_0)\]
This implies that 
\[d_{\mu}(x_{-\infty},x_0)+ d_{\mu}(x_0,x_{+\infty})= d_{\mu}(x_{-\infty},x_{+\infty}).\]
Then $x_0$ belongs to a curve $\gamma$ of minimal length for $d_\mu^\Omega$ (obtained, for example, joining a curve of minimal length from $x_{-\infty}$ to $x_0$ and one 
from  $x_0$ to $x_{+\infty}$) which exists and it is Lipschitz regular by Prop. \ref{a3}.
We now prove that 
\[spt (\gamma) \subset \A (u) \ \mbox{and} \ spt (\gamma) \subset \A(v). \]
Assume that $\gamma: [-1,1] \to \bO$  and let $t\in [-1,1]$ let $s>t$ we have
\begin{equation} \label{maxslope}
u(\gamma(s))-u(\gamma(t))= d_\mu^\Omega (\gamma(t), \gamma(s)),
\end{equation}
then, since we can choose $s$ arbitrarily close to $t$,
\[H(\gamma(t), Du)(\gamma(t))\geq \mu, \]
and, since the other inequality holds everywhere, equality holds. A consequence of this proof is that if $u(\gamma(t))=v(\gamma(t)) $ for any $t\in [-1,1]$ then 
 \[H(\gamma(t), Dv)(\gamma(t))\geq \mu. \]
We now prove that for any $t$ 
\[ S^{-}_{\mu}(u,\Omega)(\gamma(t))=S^{+}_{\mu}(u,\Omega)(\gamma(t))=u(\gamma(t)),\]
and this, by Theorem \ref{optimalsolutionsprob} will imply that $u(\gamma(t))=v(\gamma(t)) $ for any $t\in [-1,1]$ for any other minimizer $v$. By definition of $S^{+}_{\mu}(u,\Omega)$ and $S^{-}_{\mu}(u,\Omega)$, for any $t\in [-1,1]$, it holds
\begin{equation}
\begin{split}
&S^{-}_{\mu}(u,\Omega)(\gamma(t))\ge u(x_{+\infty})-d(\gamma(t),x_{+\infty})= u(\gamma(t)) \quad \mbox{and}\\
&S^{+}_{\mu}(u,\Omega)(\gamma(t))\le u(x_{-\infty})+d(x_{-\infty},\gamma(t))= u(\gamma(t)),
\end{split}
\end{equation}
where the equality follows from the \eqref{maxslope}. The proof is concluded recalling that by Theorem \ref{optimalsolutionsprob}, 
\[S^{-}_{\mu}(u,\Omega)(x)\le u(x) \le S^{+}_{\mu}(u,\Omega)(x).\]
\end{proof}

\appendix
\section{Existence of geodesics}
%%%%%%%%%%%%%%%%%%%%%%%%%%%%%%%%%
%%%%%%%%%%%%%%%%%%%%%%%%%%%%%%%%%

\begin{prop}\label{trasversal}
	Let $x,y\in\Omega$, $\gamma\in\mathsf{path}(x,y)$ and $E$ such that $\Leb^{n}(E)=0$. Then for every $\eps>0$ there exists a curve $\gamma_{\eps}$ trasversal to $E$ (i.e. $\mathcal{H}^{1}(\gamma_{\eps}((0,1))\cap E)=0$) such that 
	\begin{equation*}
	||\gamma_{\eps}-\gamma||_{\Winf((0,1))}<\eps.
	\end{equation*}
\end{prop}

\begin{proof}
	Let $g(t)\in C^{1}[0,1]$ be a non negative function such that $g(0)=g(1)=0$. For every $v\in\mathbb{R}^{n}$, we define the curve $\gamma_{v}(t)=\gamma(t)+vg(t)$ and the function $F(t,v):=\gamma(t)+vg(t)$. Let $A$ be the set of the points $(t,v)\in [0,1]\times\mathbb{R}^{n}$ such that $F(t,v)\in E$ and $A_{t}:=\{v\in\mathbb{R}^{n}: (t,v)\in A\}$. $\Leb^{n}(A_{t})=0$ for every fixed $t\in[0,1]$, since $\Leb^{n}(E)=0$. Then $\Leb^{n+1}(A)=0$,too. This implies that $A_{v}:=\{t\in[0,1]: (t,v)\in A\}$ is such that $\Leb^{1}(A_{v})=0$ for a.e. $v\in\mathbb{R}^{n}$. Let $v$ such that $\Leb^1(A)=0$. Thanks to the lipschitzianity of $\gamma_{v}$ we then have that $\mathcal{H}^{1}(\gamma_{v}(A_{v}))=0$. The claim follows noticing that $\gamma_{v}(A_{v})=\gamma_{v}([0,1])\cap E$.\\
	Finally, taking $v$ such that $|v|<{\eps}/{||g(t)||_{\Winf([0,1])}}$, it holds
	\begin{equation*}
	||\gamma_{v}-\gamma||_{\Winf((0,1))}\le|v||g||_{\Winf((0,1))}<\eps.
	\end{equation*}
\end{proof}

\begin{prop}\label{intrinsicdistance}
	Let $\Omega$ be a connected open set of $\R^{d}$ and $\lambda\ge0$. The metric space $(\Omega, d_{\lambda})$ is a length space.
\end{prop}

\begin{proof}
	For all $x,y\in\Omega$, we define the intrinsic distance $d'_{\lambda}$ associated to $d_{\lambda}$, as follows
	\begin{equation*}
	d'_{\lambda}(x,y):=\inf\{\ell(\gamma):\gamma\in\mathsf{path}_{\Omega}(x,y)\},
	\end{equation*}
	where 
	\begin{equation*}
	\ell(\gamma):=\sup\left\{\sum_{i=0}^{n-1}d_{\lambda}(\gamma(t_{i}),\gamma(t_{i+1})):0=t_{0}<t_{1}<\dots<t_{n}=1\right\}.
	\end{equation*}
	Obviously, $d_{\lambda}\le d_{\lambda}'$. We want to prove that $d_{\lambda}'\le d_{\lambda}$.
	\\Let's take $\gamma\in\mathsf{path}_{\Omega}$, we show that 
	\begin{equation*}
	\ell(\gamma)\le\int_{0}^{1}L_{\lambda}(\gamma(t),\dot{\gamma}(t))dt.
	\end{equation*}
	Let $0=t_{0}<t_{1}<\dots<t_{n}=1$ be a partition of $[0,1]$ and let $\gamma_{i}\in\mathsf{path}_{\Omega}(\gamma(t_{i}),\gamma(t_{i+1}))$ be defined by
	\begin{equation*}
	\gamma_{i}(s):=\gamma(t_{i}+s(t_{i+1}-t_{i})), \quad s\in[0,1].
	\end{equation*}
	By definition of $d_{\lambda}$ we have 
	\begin{equation*}
	\sum_{i=0}^{n-1}d_{\lambda}(\gamma(t_{i}),\gamma(t_{i+1}))\le\sum_{i=0}^{n}\int_{0}^{1}L_{\lambda}(\gamma_{i}(s),\dot{\gamma_{i}}(s))dt.
	\end{equation*}
	If $s$ is a point of differentiability of $\gamma_{i}$, then 
	\begin{equation*}
	\dot{\gamma_{i}}(s)=\dot{\gamma}(t_{i}+s(t_{i+1}-t_{i}))(t_{i+1}-t_{i})=\frac{1}{n}\dot{\gamma}(t_{i}+s(t_{i+1}-t_{i}))
	\end{equation*}
	and, by $1$-homogeneity of $L_{\lambda}$
	\begin{align*}
	&\sum_{i=0}^{n}\int_{0}^{1}L_{\lambda}(\gamma_{i}(s),\dot{\gamma_{i}}(s))dt=\sum_{i=0}^{n}\frac{1}{n}\int_{0}^{1}L_{\lambda}(\gamma(t_{i}+s(t_{i+1}-t_{i})),\dot{\gamma}(t_{i}+s(t_{i+1}-t_{i})))dt\\
	&=\sum_{i=0}^{n}\int_{t_{i}}^{t_{i+1}}L_{\lambda}(\gamma(t),\dot{\gamma}(t))dt=\int_{0}^{1}L_{\lambda}(\gamma(t),\dot{\gamma}(t))dt.
	\end{align*}
\end{proof}

\begin{prop} \label{a3}
	Let $\Omega$ be a bounded and connected open set. For any two points $x,y\in\Omega$ there exists a Lipschitz, minimizing geodesic in $\path_{\bO}(x,y)$.
\end{prop}

\begin{proof}
	First of all we notice that $\ell(\gamma)$ is lower semicontinuous w.r.t. the uniform convergence, indeed it is the sup of a family of continuous functions (indeed $d_{\lambda}$ is continuous thanks to Prop. \ref{continuitydxy}). Let 
	\[\ell:=d_\lambda(x,y):\inf\{\ell(\gamma):\gamma\in\mathsf{path}_{\Omega}(x,y)\}\]
	and consider a sequence of curves $(\gamma_n)$ such that $\ell_{n}$ converges to $\ell$. Without loss of generality we can assume that $\ell_{n}\le \ell+1$ for all $n$. We reparametrize the curves such that $\gamma_n:[0,1]\to\Omega$ such that $|\dot{\gamma_{n}}|=\ell_{n}\le\ell+1$. Since $\Omega$ is bounded we can apply Arzelà-Ascoli theorem and find a subsequence, that we will still call $(\gamma_{n})$, such that uniformly converges to a curve $\gamma$. Clearly $\path_{\bO} (x,y)$. Moreover
	\[\ell\le\ell(\gamma)\le\liminf\ell_{n}=\ell.\]
	So we have that 
	\[d_\lambda(x,y)=d'_{\lambda}(x,y)=\ell(\gamma):=\sup\left\{\sum_{i=0}^{n-1}d_{\lambda}(\gamma(t_{i}),\gamma(t_{i+1})):0=t_{0}<t_{1}<\dots<t_{n}=1\right\}. \]
\end{proof}
\begin{rmk}
	If $\pO$ is not regular and $y \in \pO$, \textit{a priori} it may happens that there are not Lipschitz curves connecting $x$ and $y$. However if we extend the definition as we did in \ref{distance}
	\begin{equation}
	d(x,y):=\inf\left\{\liminf_{n\rightarrow+\infty}d(x_{n},y_{n}) : (x_{n})_{n},(y_{n})_{n}\in \Omega^{\mathbb{N}} \ and \ x_{n}\rightarrow x,y_{n}\rightarrow y  \right\},
	\end{equation}
	if $d(x,y)<+\infty$ there exists a minimizing curve connecting the two points. To see this it is sufficient to consider $(x_{n}),(y_{n})$ two sequences converging respectively to $x$ and $y$ and such that $d(x,y)=\lim_{n\to+\infty}d(x_{n},y_{n})$. By Prop. \ref{a3}, for each $x_{n},y_{n}$ there exists a minimizing curve $\gamma_{n}$ and by similar arguments to the ones used in the proof of Prop. \ref{a3}, we find a Lipschitz curve $\gamma$, connecting $x$ to $y$, such that $d(x,y)=\ell(\gamma)$.
\end{rmk}

%%%%%%%%%%%%%%%%%%%%%%%%%%%%%%%%%%%%
\section*{Acknowledgement}
The research of the two authors is partially financed  by the {\it ``Fondi di ricerca di ateneo, ex 60 $\%$''}  of the  University of Firenze and is part of the project  {\it "Alcuni problemi di trasporto ottimo ed applicazioni"}  of the  GNAMPA-INDAM.
%%%%%%%%%%%%%%%%%%%%%

\bibliographystyle{unsrt}
%\bibliography{biblio_insiemesat}
\nocite{*}

%%%%%%%%%%%%%%%%%%%%%%%%%%%%%%%%%
%%%%%%%%%%%%%%%%%%%%%%%%%%%%%%%%%

\end{document}